\documentclass[10pt]{article}
\usepackage{lipsum}
\usepackage{amsfonts}
\usepackage{amsthm}

\usepackage{graphicx}
\usepackage{epstopdf}
\usepackage{algorithmic}
\usepackage{mathtools}
\usepackage{xcolor}
  \DeclareGraphicsExtensions{.eps,.pdf,.png,.jpg}

\usepackage{subcaption} 
\usepackage{mathrsfs}
\usepackage[space]{cite}

\usepackage{bm}
\usepackage{xspace}

\newtheorem{theorem}{Theorem}

\newtheorem{lemma}[theorem]{Lemma}
\newtheorem{remark}{Remark}

\begin{document}

\title{\bf A New Parallel-in-time Direct Inverse Method for Nonlinear Differential Equations}

\author{Nail K. Yamaleev\footnote{Corresponding author. Department of Mathematics and Statistics, Old Dominion University
Norfolk, VA 23529, USA. {\it E-mail address:} nyamalee@odu.edu} and Subhash Paudel \\
{\it\small Old Dominion University}}

\date{}






\maketitle
\begin{abstract}
We present a  new approach to parallelization of the first-order backward difference discretization (BDF1) of the time derivative in partial differential equations, such as the nonlinear heat and viscous Burgers equations. The time derivative term is discretized by using the method of lines based on the implicit BDF1 scheme, while the inviscid and viscous terms are approximated by conventional 2nd-order 3-point central discretizations of the 1st- and 2nd-order  derivatives in each spatial direction.  The global system of nonlinear discrete equations in the space-time domain is solved by the Newton method for all time levels simultaneously.  For the BDF1 discretization,  this all-at-once system at each Newton iteration is block bidiagonal,  which can be inverted directly in a blockwise manner, thus leading to a set of fully decoupled equations associated with each time level.  This allows for an efficient parallel-in-time implementation of the implicit BDF1 discretization for nonlinear differential equations.  
The proposed parallel-in-time method preserves a quadratic rate of convergence of the Newton method of the sequential BDF1 scheme, so that the computational cost of solving each block matrix in parallel is nearly identical to that of the sequential counterpart at each time step. 
Numerical results show that the new parallel-in-time BDF1 scheme provides the speedup of up to $28$ on 32 computing cores for the 2-D nonlinear partial differential equations with both smooth and discontinuous solutions.
\end{abstract}

\section{Introduction}
An accurate prediction of unsteady physical phenomena arising in various applications (e.g.,  rotorcraft and turbomachinery flows,  fluid-structure interaction, maneuvering flight conditions, weather prediction, etc.) requires a very large number of time steps, thus drastically increasing the total computational time, because conventional time integrators based on the method of lines are inherently sequential. It should be noted that parallelization of the spatial discretization alone is not enough to achieve such scalability that is required for solving these unsteady flows on modern supercomputers with millions of computing cores. For a fixed number of grid points, increasing the core count decreases the number of degrees of freedom computed on each core, which will eventually lead to the communication overhead dominating the runtime. Therefore, parallel-in-time methods offer a promising direction for achieving this goal.

Parallel time integration algorithms have been an active area of research for 60 years since they were first introduced in 
methods \cite{LMT,  DM,  Min,  Wu} which can be interpreted as iterative shooting methods,  2) space-time domain decomposition methods \cite{GHN,  MT},  3) time spectral methods for periodic flows \cite{MY,  LSLW}, 4) space-time multigrid methods \cite{Hack,  GN},  and 5) direct time-parallel methods \cite{MR,  MPW,  GH}.  

One of the most popular methods that can be directly used for both linear and nonlinear problems is the Parareal algorithm introduced in \cite{LMT}.  
For a linear system of ordinary differential equations (ODEs) with a constant coefficient matrix $A \in \mathbb{R}^{m \times m}$,  it has been proven that the Parareal method converges  if the governing equation is discretized by using an implicit $L$-stable scheme (e.g., the backward Euler scheme) and $A$ is either positive symmetric positive definite (SPD) \cite{GV} or has complex eigenvalues \cite{Wu2}.
Note,  however,  that the parallel performance of the Parareal-type methods is limited because the so-called coarse grid correction carried out on the coarse grid is inherently sequential, thus dramatically reducing the speedup of the Parareal algorithm.  As a result,  the speedup obtained with the Parareal-type methods for nonlinear problems do not usually exceed 10 regardless of the number of computing cores used \cite{OS}.  Recently,  several attempts have been made to overcome this problem by combining the Parareal algorithm with a diagonalization technique introduced in \cite{MR,MPW},  which is used to parallelize the coarse grid correction step \cite{Wu}.  Unfortunately,  the parallel efficiency of the time-parallelization methods based on diagonalization significantly deteriorates for nonlinear differential equations because  a single time-averaged Jacobian must be used for all time levels considered,  thus leading to the loss of quadratic convergence of the Newton method. Furthermore,  this time averaging effectively imposes a severe constraint on the time interval over which the averaged Jacobian provides a sufficiently accurate approximation of the true Jacobian at each time level,  which is required for convergence of Newton-type methods.  An alternative parallel-in-time approach that can be used for nonlinear problems is  space-time multigrid and multigrid in time (MGRIT) methods \cite{Hack,  HV, FFKMS}.  The space-time algorithms use the multigrid method in the entire space-time domain,  while the MGRIT methods apply multigrid only to the time dimension.  Though these methods work very well for parabolic problems with dominant viscous effects, as has been shown in  \cite{GLWYZ}, the number of iterations required for convergence and consequently the parallel efficiency of the Parareal and multigrid methods deteriorates dramatically as the physical viscosity coefficient decreases and becomes comparable to that used in practical applications governed by nonlinear partial differential equations.
For further details on the parallel efficiency and speedup provided by these methods, we refer the reader to two comprehensive reviews of time-parallel methods presented in \cite{Gan} and \cite{OS}.   As follows from these literature reviews, the existing parallel-in-time methods have not yet demonstrated parallel performance that is required for practical applications.
	
In this paper, we propose a new approach to parallelization of the method of lines for unsteady nonlinear partial differential equations of arbitrary type.  The global system of nonlinear discrete equations in the space-time domain is solved by the Newton method.  For the BDF1 scheme,  the global Jacobian  at each Newton iteration is a block bidiagonal matrix that can be directly inverted in a blockwise manner,  such that all time levels can be solved simultaneously in parallel.  Note that the computational cost of inverting each block matrix  is nearly identical to that of the sequential BDF1 scheme at each Newton iteration.   Since the decoupled and original all-at-once systems are equivalent,  the proposed parallel-in-time method preserves the quadratic convergence rate of Newton iterations,  thus providing a nearly ideal speedup for highly nonlinear nonperiodic problems with both smooth and discontinuous solutions.  This {\it Para}llelization in time based on {\it D}irect {\it In}verse method is hereafter referred to as ParaDIn.   
The proposed ParaDIn method for the implicit BDF1 discretization can be directly combined with standard spatial domain-decomposition algorithms, thus demonstrating a strong potential for obtaining a much higher speedup on large computer platforms as compared with the current state-of-the-art methods based on parallelization of the spatial discretization alone. 
The paper is organized as follows.  In section 1, we present governing equations and the baseline BDF1 scheme.  The all-at-once system of equations solved at each Newton iterations is discussed in section 3.  Then,  we present our new parallel-in-time method (ParaDIn) and study its properties in sections 4-7.  Numerical results demonstrating the parallel efficiency and scalability of the proposed method are presented in section 8.  Finally, conclusions are drawn in section 9.

\section{Governing equations and the sequential BDF1 scheme}
\label{GE}

We consider a 2-D scalar nonlinear conservation law equation of the following form:
\begin{equation}
\label{eq:CLE}
\frac{\partial u}{\partial t} + \frac{\partial f}{\partial x} + \frac{\partial g}{\partial y} = 
\frac{\partial}{\partial x}\left( \mu\frac{\partial u}{\partial x}\right) + \frac{\partial}{\partial y}\left( \mu\frac{\partial u}{\partial y}\right),  
\ \forall (x, y) \in \Omega, \ t \in (0, T_f],
\end{equation}
where $f(u)$ are $g(u)$ are inviscid fluxes associated with $x$ and $y$ coordinates,  respectively,  $\mu(u)$ is a nonlinear viscosity coefficient such that $\mu \ge 0$ $\forall u$,  and $\Omega = \left\{(x,y) | \ x_L < x < x_R, y_L < y <y_R \right\}$.  

The following two model problems are considered in the present study: nonlinear heat and Burgers equations. For the nonlinear heat equation, $f = 0$,  $g = 0$  $\forall (x,y) \in \Omega,  t\in (0, T_f]$ and $\mu(u) = \mu_0 u^2$, where $\mu_0$ is a positive constant.  In 2-D, the Burgers equations are a system of two coupled equations which are modified such that both velocity components are assumed to be equal each other in the entire space-time domain.  As a result, the system of two equations reduces to Eq. (\ref{eq:CLE}) with the following inviscid fluxes:  $f = g = \frac{u^2}2$.  For the Burgers equation, the viscosity coefficient $\mu=\mu_0$ is assumed to be a positive constant. 

The conservation law equation (\ref{eq:CLE}) is subject to the following Dirichlet boundary conditions:
\begin{equation}
\label{eq:BC}
\begin{array}{lc}
u(x_L, y, t) = b(x_L, y, t), \ \ \ & u(x_R, y, t) = b(x_R,  y, t), \\
u(x, y_L, t) = b(x, y_L, t), \ \ \ & u(x, y_R, t) = b(x,  y_R, t),
\end{array}
\end{equation}
and the initial condition: 
$$ u(x, y, 0) = u_0(x, y, 0),$$
where it is assumed that  $b$ and $u_0$ are bounded in $L_2 \cap L_{\infty}$ and contain data so that Eqs. (\ref{eq:CLE}-\ref{eq:BC}) are well posed.

The governing equation (\ref{eq:CLE}) is approximated on a uniform rectangular grid generated in the domain $\Omega$ with $N_x$ and $N_y$ grid intervals along the $x$ and $y$ coordinates,  respectively.  The grid points are enumerated in a conventional way along the $x$-axis starting from the bottom boundary $j=0$ and going up to the top boundary $j=N_y$.  Discretizing the time derivative term by using the method of lines based on the implicit BDF1 scheme and the viscous and inviscid terms by the conventional 2nd-order 3-point discrete Laplace operator and the 2nd-order central finite difference scheme in each spatial direction,  respectively,  we have
\begin{equation}
\label{eq:FDS}
\begin{split}
 \frac{u_{j,i}^n-u_{j,i}^{n-1}}{\tau_n} + &\frac{f_{j,i+1}^n - f_{j,i-1}^n }{2 h_x} + \frac{f_{j+1,i}^n - f_{j-1,i}^n }{2 h_y} \\
                                                              = & \frac{\mu_{j, i+\frac 12} \left(u_{j,i+1}^n - u_{j,i}^n\right) -  \mu_{j, i-\frac 12} \left(u_{j,i}^n - u_{j,i-1}^n\right)}{h^2_x} \\
                                                              + & \frac{\mu_{j+\frac 12,i} \left(u_{j+1,i}^n - u_{j,i}^n\right) -  \mu_{j-\frac 12,i} \left(u_{j,i}^n - u_{j-1,i}^n\right)}{h^2_y},
\end{split}
\end{equation}
where $\tau_n = t^n - t^{n-1}$ is a time step size,  $h_x$ and $h_y$ are grid spacings in $x$ and $y$, respectively.  
Using the Newton method at each time step, the nonlinear discrete equations (\ref{eq:FDS}) are solved sequentially starting from the time level $n=1$ and marching forward in time until $n = N_t$.  Since the solution $u^{n-1}$ at the previous time level is required to advance the solution to the next time level,  this time integration method is inherently sequential in time.

\section{Global-in-time formulation of the BDF1 scheme}
\label{GIT}
In contrast to the conventional time-marching method that solves the nonlinear discrete equations (\ref{eq:FDS}) sequentially in time,  the new time-parallel BDF1 method considers Eq. (\ref{eq:FDS}) as a single global space-time system of equations:
\begin{equation}
\label{eq:GIT}
\frac{{\bf u}^n -{\bf u}^{n-1}}{\tau_n} + {\bf F}\left({\bf u}^n\right) = {\bf q}^n,  \hspace{0.5cm} n = 1, \dots, N_t,
\end{equation}
where the solution vector is given by ${\bf u}^n = \left[u_{1,1}^n,  \dots , u_{1,N_x}^n, \dots, u_{N_y,  N_x}^n \right]^T$, ${\bf F}$ represents the nonlinear discrete spatial operator associated with the inviscid and viscous terms,   and the vector ${\bf q}$ includes the contribution from the initial and boundary conditions and the source term.  Multiplying Eq. (\ref{eq:GIT}) by $\tau_n$ and using the Newton method to solve this global-in-time system of nonlinear equations, we have
\begin{equation}
\label{eq:LocalNewton}
\left[ I + \tau_n \left(\frac{\partial{\bf F}}{\partial {\bf u}}\right)^n_k \right]\Delta{\bf u}^n  - \Delta{\bf u}^{n-1} = 
-\tau_n \left(\frac{{\bf u}^n_k -{\bf u}^{n-1}_k}{\tau_n} + {\bf F}\left({\bf u}^n_k\right) - {\bf q}^n_k\right), 
\end{equation}
where ${\bf u}_{k+1}^n = {\bf u}_k^n + \Delta{\bf u}^n$,  ${\bf u}_k^n$ is the solution vector of length $N_s=N_x N_y$ on $n$-th time level at $k$-th Newton iteration, and $I$ is the $N_s\times N_s$ indentity matrix.  Note that Eq. (5) is solved for all time levels simultaneously,  i.e.,  it is solved for the following global vector ${\bf U} = \left[{\bf u}^1, {\bf u}^2, \dots,  {\bf u}^{N_t}\right]^T$, where ${\bf u}^i \in \mathbb{R}^{N_s} \forall i$. The above system of linear equations at each Newton iteration can be recast in a block--matrix form as follows:
\begin{equation}
\label{eq:GlobalNewton}
\left[
\begin{array}{cccc}
A_1      &   0        & \dots   &       0      \\
-I         & A_2      & \ddots &  \vdots   \\
\vdots  & \ddots  & \ddots &      0       \\
0         & \dots    &     -I     &  A_{N_t} \\
\end{array}
\right]
\left[
\begin{array}{c}
\Delta{\bf u}^1         \\
\Delta{\bf u}^2         \\
\vdots                        \\
\Delta{\bf u}^{N_t}   \\
\end{array}
\right] = 
\left[
\begin{array}{c}
{\bf r}^1         \\
{\bf r}^2         \\
\vdots             \\
{\bf r}^{N_t}   \\
\end{array}
\right],
\end{equation}
where $A_i = I + \tau_i \left(\frac{\partial{\bf F}}{\partial {\bf u}}\right)^i_k$, $i=1,\dots, N_t$ are nonsingular $N_s \times N_s$ matrices and ${\bf r}^n$ is a residual vector associated with the $n$-th time level,  which is equal to the right-hand side of Eq. (\ref{eq:LocalNewton}).  The system of equations (\ref{eq:GlobalNewton}) is coupled in time only due to the subdiagonal of the left-hand-side matrix that couples vectors $\Delta{u}^n$ and $\Delta{u}^{n-1}$ at two neighboring time levels, thus making it impossible to solve each spatial equation simultaneously in time.

It should be noted that  Eq.  (\ref{eq:LocalNewton}) at each Newton iteration can be solved iteratively in parallel by using either the block-Jacobi or  block-Gauss-Seidel methods \cite{SN, Wom}.  Note, however,  that these methods converge very slowly as has been proven in \cite{DMD}, thus dramatically reducing or completely negating the overall parallel speedup of the method.  We propose to overcome this problem by fully decoupling equations (\ref{eq:LocalNewton}) and solving them directly in parallel. This new approach referred to as ParaDIn is presented next.

\section{New parallel-in-time BDF1 scheme}
\label{PITBDF1}

To solve Eq. (\ref{eq:GlobalNewton}) in parallel for all time levels,  we propose to fully decouple this block bi-diagonal system of equations such that each block associated with the spatial operator at a given time level can be solved independently of each other. The main idea of the proposed method is based on the observation that the exact inverse of the block bi-diagonal matrix in the left-hand side of Eq.  (\ref{eq:GlobalNewton}) is a lower block-triangular matrix which can be calculated analytically.  Multiplying the right-hand side ${\bf R}=\left[{\bf r}_1, \dots, {\bf r}_{N_t} \right]^T$ by this inverse,  the system of equations Eq.  (\ref{eq:GlobalNewton}) can be fully decoupled.  Indeed,  assuming that $A_i$ is nonsingular for all $i=1, \dots, N_t$ and its inverse is denoted as $A_i^{-1}$,  the solution of Eq.  (\ref{eq:GlobalNewton}) can be written as follows:
\begin{equation}
\label{eq:Inverse}
\left[
\begin{array}{c}
\Delta{\bf u}^1         \\
\Delta{\bf u}^2         \\
\vdots                        \\
\Delta{\bf u}^{N_t}   \\
\end{array}
\right] = 
\left[
\begin{array}{cccc}
A_1^{-1}                                           &   0             & \dots                                                          &       0      \\
 \prod\limits_{i=2}^{1} A_i^{-1}      & A_2^{-1}  & \ddots                                                        &  \vdots   \\
\vdots                                               &\ddots       & \ddots                                                        &      0       \\
 \prod\limits_{i=N_t}^{1} A_i^{-1}  & \dots        &  \prod\limits^{N_t-1}_{i=N_t} A_i^{-1}    &  A_{N_t}^{-1} \\
\end{array}
\right]
\left[
\begin{array}{c}
{\bf r}^1         \\
{\bf r}^2         \\
\vdots             \\
{\bf r}^{N_t}   \\
\end{array}
\right],
\end{equation}
where $ \prod\limits_{i=N_t}^{1} A_i^{-1}  = A_{N_t}^{-1} \dots A_1^{-1}$.
Formally,  the space-time system of linear equations (\ref{eq:Inverse}) is fully decoupled in time and can be integrated in parallel for all time levels. 

\begin{remark}
\label{re:Commute}
Note that the order of matrix multiplication in Eq. (\ref{eq:Inverse}) is important, because the Jacobian matrices $A_i$, $i=1,\dots,N_t$ depend on the corresponding solution vectors ${\bf u}_i$ and do not in general commute with each other,  i.e., 
$$
A_i A_j \ne A_j A_i, \quad \forall i \ne j. 
$$
\end{remark}

Though the system of equations (\ref{eq:Inverse}) is fully decoupled,  the inverse Jacobians $A_i^{-1}$, $i=1,\dots,N_t$ are not readily available for 
discretized nonlinear conservation law equations such as Eq.  (\ref{eq:FDS}).  We propose to overcome this problem by decoupling Eq. (\ref{eq:GlobalNewton}) in an alternative form that does not explicitly require the local inverse Jacobians,  which is presented in the following theorem.
\begin{theorem}
\label{th:Decouple}
If matrices $A_i$, $i=1,\dots,N_t$  in Eq. (\ref{eq:GlobalNewton}) are nonsingular,  i.e., $\det A_i \ne 0, \ \forall i$, then the following system of equations 
\begin{equation}
\label{eq:Dec}
\left\{
\begin{array}{rl}
A_1 \Delta {\bf v}^1        = & {\bf r_1} \\
A_1 A_2\Delta {\bf v}^2 = &A_1 {\bf r}_2 + {\bf r_1} \\
\vdots & \\
 \prod\limits_{i=1}^{N_t} A_i \Delta{\bf v}^{N_t} = &  \sum\limits_{j=2}^{N_t}  \prod\limits_{i=1}^{j-1} A_i {\bf r}_j  + {\bf r}_1 \\
\end{array}
\right.
\end{equation}
has a unique solution that is identical to that of Eq. (\ref{eq:GlobalNewton}).
\end{theorem}
\begin{proof}
The existence and uniqueness of the solution of Eq. (\ref{eq:Dec}) follow immediately from the fact that $\det\left( \prod\limits_{i=1}^{j} A_i \right) \ne 0$
for $j=1, \dots, N_t$  which is a direct consequence of $\det A_i \ne 0$ for $i=1, \dots,  N_t$.  

Now,  let us prove that the solutions of Eqs. (\ref{eq:GlobalNewton}) and (\ref{eq:Dec}) are identical each other.  
We begin by noting that the first equations in both systems are identical to each other and fully decoupled form the remaining equations, which implies that $\Delta{\bf v}_1  \equiv \Delta{\bf u}_1$. 
 Moving the $\Delta{\bf u}_1$ in  the 2nd equation of Eq. (\ref{eq:GlobalNewton}) to the right-hand side and multiplying both sides of the equation by the nonsingular matrix $A_1$, we have
$$
A_1 A_2 \Delta{\bf u}_2 = A_1 {\bf r}_2 + A_1 \Delta {\bf u}_1.
$$
Substituting the first equation in Eq. (\ref{eq:GlobalNewton}) into the right-hand side of the above equation yields
\begin{equation}
\label{eq:Eq2}
A_1 A_2 \Delta{\bf u}_2 = A_1 {\bf r}_2 + {\bf r}_1,
\end{equation}
which is identical to the second equation in the system of equations (\ref{eq:Dec}).  Moving the $\Delta{\bf u}_2$ in  the 3rd equation of Eq. (\ref{eq:GlobalNewton}) to the right-hand side,  multiplying by $A_1 A_2$, and substituting  Eq. (\ref{eq:Eq2}), into this equation, we obtain the following equation:
$$
A_1 A_2 A_3\Delta{\bf u}_3 = A_1 A_2{\bf r}_3 + A_1{\bf r}_2 + {\bf r}_1
$$
which is identical to the 3rd equation in Eq. (\ref{eq:Dec}).  Repeating this procedure recursively $N_t-1$ times, we recover all the equations in Eq. (\ref{eq:Dec}).  Since at each step of this procedure, we performed the operations that do not change the solution of the original system of equations (\ref{eq:GlobalNewton}), the solution of Eq. (\ref{eq:Dec}) is identical to the solution of (\ref{eq:GlobalNewton}), i.e.,  $\Delta{\bf v}_i \equiv \Delta{\bf u}_i$ $\forall i = 1,\dots, N_t$.
\end{proof}

A key advantage of Eq. (\ref{eq:Dec}) as compared with its counterpart (Eq.(\ref{eq:GlobalNewton})) is that this system of equations is fully decoupled and can be solved in parallel for all $N_t$ time levels at once.  This allows for an efficient parallel implementation of the method of lines based on the BDF1 discretization.  Since the systems of equations (\ref{eq:GlobalNewton}) and (\ref{eq:Dec}) are equivalent to each other and have identical solutions as follows from Theorem \ref{th:Decouple}, the convergence rates of Newton iterations associated with both systems of equations are identical to each other as well.  Therefore,  if the Newton method associated with Eq. (\ref{eq:GlobalNewton}) converges quadratically, then the new parallel-in-time method based on Eq. (\ref{eq:Dec}) preserves this optimal convergence rate.
Note that at each time level,  the matrix on the left-hand side of each equation  in Eq. (\ref{eq:Dec}) has the same size as the system of linear equations at the corresponding time step at each Newton iteration of the sequential BDF1 method, which is solved to advance the solution over one time step. Therefore, the computational cost of solving each local problem in Eq. (\ref{eq:Dec}) is nearly identical to that of the sequential BDF1 scheme at each Newton iteration at the corresponding time step,  if the same direct solver is used for solving both systems of linear equations. Another advantage of the proposed parallel-in-time integration scheme is that it preserves the original spatial discretization and can be directly combined with standard spatial domain-decomposition algorithms, thus promising a much higher speedup on large computer platforms as compared with the current state-of-the-art methods based on parallelization of the spatial discretization alone. Furthermore, the new BDF1 scheme given by Eq. (\ref{eq:Dec}) provides the highest level of parallelism in time because each time level can be computed on its own computing core in parallel. 

\begin{remark}
Along with the 2nd-order linear finite difference operators used to discretize the spatial derivatives in the present study,  the proposed parallel-in-time BDF1 scheme can be straightforwardly extended to other spatial discretizations (finite element,  finite volume,  spectral collocation methods, etc.) without any modifications.
\end{remark}

The Newton method for the all-at-once space-time discrete equations (\ref{eq:Dec}) requires an initial guess for the entire vector  ${\bf U} = \left[{\bf u}^1, {\bf u}^2, \dots, {\bf u}^{N_t }\right]^T$, where ${\bf u}^{i}\in \mathbb{R}^{N_s}$ $\forall i$.  Since the Newton method is sensitive to the initial guess, in the present analysis, we initialize the solution by solving Eq. (\ref{eq:FDS}) sequentially on a coarse grid and interpolating this coarse-grid-solution  to the original fine mesh by using a 1-D cubic spline along each spatial coordinate.  Note that the original grid is coarsened in both each spatial direction and time by the same factor $c_f$.  
Though this strategy effectively imposes an upper bound on the maximum possible speedup that can be obtained with the ParaDIn method,  the bound is not too restrictive.  Indeed, the upper bound is $O(c_f^p )$ with $p \ge 3$ in 2-D and $p \ge 4$ in 3D,  because for realistic problems, the computational cost associated with the solution of the large system of linear equations at each time level is practically always superlinear in $N_s$,  where $N_s$ is the total number of grid points of the spatial grid.  For example for $c_f=4$, this bound is of the order of $O(c_f^3 )=O(64)$, which is well above than most, if not all, parallel speedups in time which have been reported in the literature for 2-D nonlinear partial differential equations.

\section{Parallel algorithm for calculating  left- and right-hand sides of Eq. (\ref{eq:Dec})}
\label{RLHS}
To construct an efficient parallel-in-time method, the computational cost of calculating the left- and right-hand sides of each equation in Eq. (\ref{eq:Dec}) must be significantly lower than that of solving the corresponding system of linear equations at each time level. 
Note that the left- and right-hand sides of Eq. (\ref{eq:Dec}) have the following recursive property:
$$
\begin{array}{l}
P^{n+1} = P^{n} A_{n+1} \\
\tilde{\bf r}^{n+1} = P^{n}{\bf  r}^{n+1} + \tilde{\bf r}^n,
\end{array}
\quad {\rm for} \ n = 1,\dots N_t
$$
where 
$P^1 = A_1$, and $\tilde{\bf r}^1 = {\bf r}^1$.
Using this recursive property, we propose the following algorithm for computing all $N_t$ left- and right-hand sides in Eq. (\ref{eq:Dec}) which can be performed in parallel.  First,  the solution vector ${\bf u}_1$ is sent to all computing cores. Then, we partition the matrix $A_1$ into $N_t$ rectangular matrices,  such that each submatrix containing $m=N_s/N_t$ rows of $A_1$ is formed on the corresponding computing core.  In other words,  the rows from $1$ to $m$ are computed on a core $C_1$,  rows from $m+1$ to $2m$ are computed on a core $C_2$ and so on up to the last computing core $C_{N_t}$.  As a result,  each core has the corresponding batch of rows of the original matrix $A_1$.  To take into account the sparsity of the Jacobian matrices $A_1, \dots,  A_{N_t}$ and reduce the computational cost,  we first multiply nonzero entries in the rows of the $A_1$ matrix by the corresponding nonzero entries in each column of $A_2$.  After this multiplication, each core has the corresponding $m$ rows of the product matrix $A_1 A_2$,  which are then sent to the core $C_2$ to form the entire product matrix $A_1 A_2$.  Similarly,  we can calculate the product of $A_1$ and $m$ rows of the vector ${\bf r}_2$ on each core.  Repeating this multiplication procedure recursively $l$ times,  the corresponding rows of the product matrix $P^l = \prod\limits_{i=1}^l A_i $ and the right-hand-side 
$\tilde{\bf r}^l$ of Eq. (\ref{eq:Dec}) on each computing core are calculated as follows:
\begin{equation}
\label{eq:Prod}
\begin{split}
p^{l}_{ij} &= \sum\limits_{k\in K_{ls}} p^{l-1}_{ik} a^{l}_{kj},                            \quad l=2,\dots, N_t, \ i = (s-1) m + 1, \dots, s m \\
\tilde{r}^{l}_{i}  &= \sum\limits_{k\in K_{l}} p^{l-1}_{ik} r^{l}_{k} + \tilde{r}^{l-1}_{i},  \\
\end{split}
\end{equation}
where $\tilde{r}^{l}_{i}$ and $r^{l}_{k}$ are entries of the right-hand sides in the $l$-th equation in Eqs. (\ref{eq:Dec}) and (\ref{eq:GlobalNewton}), respectively,  and $K_{ls}$ is a set of indices such that both the entries of the product  matrix $P^{l-1}$ and the entries of the Jacobian matrix $A_l$ are nonzero.  
The above algorithm can be  summarized as follows.
\newline\newline
\noindent{\bf Algorithm 1}

\begin{tabular}{l}
For $l=1,\dots, N_t$  do\\
1) Send ${\bf u}^l$  and calculate rows $\left\{a^l_{ij}\right\}_{i={(s-1)}m+1}^{s m}$ on $C_s$ for $s=1,\dots, N_t$ \\
2) Compute $\left\{p^l_{ij}\right\}_{i={(s-1)}m+1}^{s m}$ and $\left\{\tilde{r}^l_{i}\right\}_{i={(s-1)}m+1}^{s m}$  by using Eq. (\ref{eq:Prod})\\
in parallel on $C_s$ for $s=1,\dots, N_t$  \\
3) Send $\left\{p^l_{ij}\right\}_{i={(s-1)}m+1}^{s m}$ to $C_s$ for all $l \ne s$
\end{tabular}
\newline

For the numerical scheme (Eq.~(\ref{eq:FDS})) considered in the present study, the number of nonzero entries  in each row/column of $A_1, \dots, A_{N_t}$ is less than or equal to $d=5$. 
Therefore,  let us show that the computational cost of the above parallel algorithm for calculating the left- and right-hand sides of all equations in the system (\ref{eq:Dec}) is $O(N_t^2 N_s)$, i.e., it is just linear in $N_s$ if $N_t \ll N_s^{1/2}$.  We begin by proving the following lemma.
\begin{lemma}
\label{lm}
The product of two $N_s \times N_s$ matrices $X$ and $Y$ where the latter is a sparse matrix whose each column contains no more than $d$ nonzero elements can be performed in $2  nnz(Z) d$ operations, where $nnz(Z)$ is the number of all nonzero entries in $Z=X Y$.
\end{lemma}
\begin{proof}
Let $z_{ij}$ be a nonzero entry of $Z$, which is computed as a dot product of the $i$-th row of $X$ and $j$-th column of $Y$, i.e.,
$$
z_{ij} = \sum\limits_{k=1}^{N_s} x_{ik} y_{kj} .
$$ 
Since in each column of $Y$ there are only $d$ or less nonzero entries, the above sum contains no more than $d$ nonzero terms, which can be computed in $2d$ operations  assuming that the addition and multiplication are counted as identical operations. Taking into account that the same number of operations is needed for each of $nnz(Z)$ entries in $Z$, the total number of operations is $2 nnz(Z) d$.
\end{proof}

With this lemma, we can now prove the following theorem that provides an estimate of the total computational cost required for computing all product matrices in Eq. (\ref{eq:Dec}).
\begin{theorem}
\label{th:Prod}
The total number of operations for computing all product matrices $P^l=\prod\limits_{i=1}^l A_i, l=1, \dots, N_t$ by using Algorithm 1 is $O(M^2 N_s)$, where $M=\min\left(N_t, N_x\right)$ and $A_1, \dots, A_{N_t}$ are $N_s\times N_s$ Jacobian matrices of the numerical scheme given by Eq. (\ref{eq:LocalNewton}).
\end{theorem}
\begin{proof}
First, let us show that if $N_t < N_x =O(N_s^{1/2})$, i.e., $M=N_t$, the total number of operations of Algorithm 1 is $O(N_t^2 N_s)$.  Indeed,  for the central  scheme given by Eq. (\ref{eq:FDS}), each Jacobian matrix $A_i$ contains 5 nonzero diagonals.  Furthermore, all Jacobian matrices have the same pattern of nonzero entries, i.e.,  if $a^k_{ij}\ne 0$, then $a^l_{ij}\ne 0$ and vice versa if $a^k_{ij}=0$, then $a^l_{ij}=0 \ \forall k, l=1,\dots,N_t$.   If $N_t < N_x$,  it can be readily shown by the direct multiplication that the number of diagonals in  $P^{l} =\prod\limits_{i=1}^{l} A_i$ grows as an arithmetic sequence, where the number of elements in the sequence is $l$ and its common difference is $O(5)$.
Therefore,  the total number of nonzero diagonals in the product matrix $P^{N_t}$ is equal to the sum of this arithmetic sequence which is $O(N_t^2)$, where $N_t$ is the total number of elements in the sequence.  Using Lemma \ref{lm} and taking into account that only $N_s/N_t$ rows of $P^l$ are computed on each core, the number of operations performed by each core for calculating these rows of $P^l$ is of the order of $O(N_t^2 \frac{N_s}{N_t})$.  Since this multiplication procedure is repeated $N_t$ times, the total computational cost for calculating all product matrices $P^1, \dots, P^{N_t}$ is $N_t O(N_t^2 \frac{N_s}{N_t}) = O(N^2_t N_s)$.

For $N_t \ge N_x=O(N_s^{1/2})$,  i.e., $M=O(N_s^{1/2})$,  the worst case scenario is when the number of matrices  $l$ in the product $P^l=\prod\limits_{i=1}^l A_i$ is so high that $P^l$ is a full matrix.  Therefore, according to Lemma \ref{lm}, the multiplication of $N_s/N_t$ rows of the full matrix $P^l$ by $A_{l+1}$ that contains at most $5$ nonzero entries in each column requires $O(N_s \frac{N_s}{N_t})$ operations. Again, repeating this multiplication recursively $N_t$ times, the total number of operations required for computing $P^1, \dots, P^{N_t}$ is $N_t O(N_s \frac{N_s}{N_t}) = O(N_s^2)$.
\end{proof}

\begin{remark}
The number of operations needed for computing all  right-hand sides ${\bf r}_1, \dots {\bf r}_{N_t}$ is also $O(M^2 N_s)$, where $M=\min\left(N_t, N_x\right)$.  Indeed,  for $N_t < N_x=O(N_s^{1/2})$, each row of the product matrix $P^{n}$ contains at most $O(n^2)$ nonzero entries (see Theorem \ref{th:Prod}).  Therefore, the dot product of the nonzero entries in $N_s/N_t$ rows of $P^{n}$ by the vector ${\bf r}_n$ can be performed in $O(n^2 \frac{N_s}{N_t})$ operations.  Repeating this procedure recursively $N_t$ times, the total number of operations for computing ${\bf r}_1, \dots {\bf r}_{N_t}$ is $O(N^2_t N_s$). The case $N_t \ge N_s^{1/2}$ can be proven similarly.
\end{remark}
\begin{remark}
For $N_t < N_x=N_y$, the exact number of nonzero diagonals in a product matrix $P^n=\prod\limits_{i=1}^n A_i$ for the 5-point central scheme given by Eq.  (\ref{eq:FDS}) is equal to $2n^2 + 2n +1$.
\end{remark}


\section{Computational cost and speedup}
Let us evaluate a speedup that can be  obtained with the new ParaDIn method for the implicit BDF1 scheme outlined in sections \ref{PITBDF1} and \ref{RLHS}.  For the serial sequential algorithm, the overall computational cost can be evaluated as follows:
$$
W_{\rm seq} = k_{\rm seq}  N_t  \left[O(N_s) + W_{\rm sol}\right] 
$$
where $W_{\rm sol}$ is the computational cost of solving a system of $N_s$ linear equations at each Newton iteration which strongly depends on a particular solver used, the first term in the square brackets represents the number of operations required to form the Jacobian matrix at each Newton iteration at each time step,  and $k_{\rm seq}$ is an average number of Newton iterations per time step of the sequential algorithm.  Note that the number of operations for computing each Jacobian matrix is of the order of $O(N_s)$,  because this Jacobian matrix contains only 5 nonzero diagonals for the central scheme given by Eq. (\ref{eq:FDS}).  

For the proposed  ParaDIn method,  each time level is computed on its own computing core in parallel. Therefore, the computational time of the parallel-in-time BDF1 method is equal to the sum of the time required to compute a single time step on all computing cores and the communication time between the cores. The computation of each time step on an individual core includes, the cost of forming the corresponding flux Jacobian matrix $A_i$ and computing the left- and right-hand sides in Eq. (\ref{eq:Dec}),  the cost of solving the system of $N_s$ linear equations at each time level,  and the communication time between the cores.  Note that these computational costs should be multiplied by the number of Newton iterations required to drive the residual below a user-specified tolerance.   
In the foregoing section, it has been proven that all $N_t$ left- and right-hand sides in Eq. (\ref{eq:Dec}) can be evaluated in parallel in $O(N_t^2 N_s)$ operations if $N_t < O(N_s^{1/2})$.  The computational cost of solving the system of $N_s$ linear equations at each time level in  Eq. (\ref{eq:Dec}) is assumed to be identical to that of the sequential algorithm $W_{\rm sol}$, which strongly depends on a linear solver used.  In the present analysis, we use the same direct solver for banded matrices for both the parallel-in-time and sequential BDF1 methods to eliminate its effect on the comparison of the ParaDIn and sequential methods.  As we discuss in section \ref{PITBDF1},  the overall computational time also includes the cost of computing an initial guess for all time levels.   For 2-D problems considered in the present paper, this computational cost can be estimated as follows: $ k_{\rm seq} W_{\rm sol}/c_f^p$, where $c_f$ is a space-time coarsening factor and $p$ is a constant such that $p\ge 3$ and $p \ge4$ in two and three spatial dimensions, respectively.  Combining these computational costs together, the total number of operations needed to solve the nonlinear discrete equations (\ref{eq:GIT}) in parallel is given by
\begin{equation}
W_{\rm par} = k_{\rm seq}N_t\frac{W_{\rm sol}}{c_f^p} + k_{\rm par} \left[ O(N_s) +  O(N_t^2 N_s) + W_{\rm sol}  + W_{\rm com}\right],
\end{equation}
where $W_{\rm com}$ is the communication time between computing cores and $k_{\rm par}$ is the number of Newton iterations of the parallel-in-time method.
As a result, the speedup that can be obtained using the ParaDIn--BDF1 scheme as compared with its sequential counterpart is given by
\begin{equation}
S = \frac{W_{\rm seq}}{W_{\rm par}} = \frac{k_{\rm seq}  N_t  \left[O(N_s) + W_{\rm sol}\right] }{k_{\rm seq} N_t \frac{W_{\rm sol}}{c_f^p} + k_{\rm par} \left[ O(N_s) +  O(N_t^2 N_s) + W_{\rm sol}  + W_{\rm com}\right]}
\end{equation}
For realistic problems described by highly nonlinear partial differential equations (e.g., the Navier-Stokes equations),  it is safe to assume that the computational cost of solving the large system of linear equations at each Newton iteration is dominant, 
so that $W_{\rm sol} \gg O(N^2_t N_s) \gg O(N_s)$ and significantly higher than the communication cost, i.e., $W_{\rm sol} \gg W_{\rm com}$, if $N_t$ is not too large,  e.g., $N_t \ll O(N_s^{1/2})$.  As follows from Theorem \ref{th:Prod}, the solution of Eq. (\ref{eq:Dec}) is identical to that of the original equations (\ref{eq:GlobalNewton}), thus indicating that the Newton iterations of the ParaDIn method converge at the same rate as those of the original sequential algorithm, i.e., $k_{\rm par} \approx k_{\rm seq}$.
Therefore,  a speedup provided by ParaDIn method can be estimated as follows:
\begin{equation}
\label{eq:speedup}
S = \frac{N_t}{\frac {N_t}{c_f^p}  + 1}.
\end{equation}
For $N_t \ll c_f^p$, we can conclude from Eq. (\ref{eq:speedup}) that a nearly ideal speedup of $N_t$ can be achieved with the proposed ParaDIn method.
%


\section{Influence of the condition number }
\label{cond}
As follows from Theorem \ref{th:Prod}, the following constraint $N_t \ll O(N_s^{1/2})$ should be imposed on the total number of time steps to make the computational cost of forming the left- and righ-hand sides of Eq. (\ref{eq:Dec}) linear in the number of degrees of freedom used for discretizing the spatial derivatives.  It should be noted,  however, that this is not the only constraint that should be imposed on the total number of time steps $N_t$ which is equal to the total number of computing cores used for the parallel-in-time discretization of the time derivative.  Indeed, the left-hand side of Eq. (\ref{eq:Dec}) is the product of $n$ Jacobian matrices $\prod\limits_{i=1}^n A_i$ whose condition number may become very large as $n$ increases.  As a result, even a direct linear solver may introduce an O(1) error in the solution of the linear system of equations: $\prod\limits_{i=1}^n A_i \Delta {\bf u}^n = \tilde{\bf r}^n$, if its condition number is comparable or larger than $O(\epsilon^{-1})$, where $\epsilon$ is a roundoff error that accumulates during solving this system of linear equations.  In the present analysis,   a direct solver for banded matrices is used to solve the system of linear equations at each time level.  Therefore, the total roundoff  error introduced by the solver can be estimated as follows: $\epsilon=\epsilon_{\rm rof} O(N_s^2)$, where $\epsilon_{\rm rof}$ is the machine roundoff error which is usually $O(10^{-16})$ and the second term represents the number of operations performed to solve this system of linear equations.

Let us now evaluate the maximum number of time steps that can be used in Eq. (\ref{eq:Dec}) for the 2-D linear heat equation until the condition number becomes comparable with  $\epsilon^{-1}$.  It is well known that the Jacobian of the 5-point Laplacian operator is a symmetric positive definite matrix whose eigenvalues are given by
\begin{equation}
\lambda_{k_x k_y} = \frac{4}{h_x^2}\sin^2{\frac{\pi k_x h_x}2} + \frac{4}{h_y^2}\sin^2{\frac{\pi k_y h_y}2},
\end{equation}
where $k_x=1,\dots,N_x-1$ and $k_y=1,\dots,N_y-1$.  The eigenvalues of $A_i =  I + \tau_i \mu \Delta_h$, where $\Delta_h$ is the Jacobian matrix of the discrete 5-point Laplace operator, can be expressed as follows: $1 + \tau \mu \lambda_{k_x k_y}$.  Assuming that $h_x = \frac1{N_x}=h_y=\frac1{N_y}=h$,  $\tau_1 =\dots =\tau_{N_t} =  \tau=\frac{1}{N_t}$ and all matrices $A_1,\dots,A_{N_t}$ are identical to each other,  the condition number of $\prod\limits_{i=1}^{N_t} A_i $ can be estimated as:
$$
\mathcal{K} = \left( \frac{\max\limits_{k_x, k_y}\lambda_{k_x k_y} }{\min\limits_{k_x, k_y}\lambda_{k_x k_y}}\right)^{N_t} = \left(\frac{1+\frac{8\mu\tau}{h^2}}{1+ 2\pi^2 \mu\tau}\right)^{N_t}.
$$
Setting the condition number $\mathcal{K}$ equal to the reciprocal of the total roundoff error $\epsilon^{-1}$, we have the following inequality for $N_t$:
\begin{equation}
\label{cond}
\left(\frac{1+\frac{8\mu}{N_t h^2}}{1+ \frac{2\pi^2 \mu}{N_t}}\right)^{N_t} <  \frac{c h^4}{\epsilon_{\rm rof}} ,
\end{equation}
where $c$ is a constant independent of $h$.
As follows from the above estimate, the maximum number of time steps that can be integrated by the proposed parallel-in-time method without significant accumulation of the roundoff error depends on the spatial grid spacing $h$ and the physical viscosity $\mu$.  For example for $N_x=N_y=64,  \mu=10^{-3},  c=1$, the maximum number of time steps $N_t$ given by Eq. (\ref{cond}) is equal to 22.  Note that $N_t$ given my Eq.(\ref{cond}) increases  as the viscosity coefficient $\mu$ decreases. Our numerical results presented in the next section qualitatively corroborate this estimate.  
To partially circumvent the influence of the condition number,  a simple diagonal preconditioner,  $\mathscr{P} = {\rm diag}[p^n_{11}, \dots,  p^n_{N_s N_s}]$, where $p^n_{ii}$ are diagonal entries of the product matrix $P^n$, is used in the present analysis.   Constructing more efficient preconditioners for the ParaDIn method is the subject of ongoing research, which will be presented in our future paper.

\section{Numerical results}
\label{results}
We test the proposed ParaDIn-BDF1 scheme on standard benchmark problems including the 2-D nonlinear heat and Burgers equations with smooth and discontinuous solutions. For all numerical experiments presented herein, we set $N_x=N_y$ and use uniform grids in time,  i.e., $\tau_1 = \tau_2 = \dots = \tau_{N_t }$. The accuracy of the ParaDIn method is verified against the corresponding sequential scheme and exact solutions.  The system of linear equations at each time level is solved by using a standard direct solver for banded matrices without pivoting.  To make a fair comparison between the sequential and parallel-in-time BDF1 schemes, the same direct solver for banded matrices is used for both methods.  {It is beyond the scope of the current study to conduct a detailed comparison of the performance of various solvers for the new parallel-in-time BDF1 scheme.}
The ParaDIn-BDF1 scheme has been parallelized only in time, such that the number of computing cores used is precisely equal to the total number of time steps $N_t$, i.e., each time step is calculated on one core regardless of the number of degrees of freedom used for discretization of the spatial derivatives.
Note that $L_1$,  $L_2$, and $L_{\infty}$ error norms presented in this section are evaluated over the entire space-time domain.  For all test problems presented in this section, we run the ParaDIn and sequential BDF1 schemes with identical input parameters and on the same space-time grids and computing cores.
\begin{table}[!h]
\begin{center}
\begin{tabular}{cccccc}
\hline
                                                &  Sequential BDF1 &   & \vline &  ParaDIn BDF1 &  \\
\hline
 $N_t\times N_x \times N_y$ & $ L_{\infty}$ error & $L_{\infty}$ rate  &\vline& $ L_{\infty}$ error & $L_{\infty}$ rate \\
\hline
   $8\times 8\times 8     $   &  $2.22\times 10^{-4}$  & $ - $      & &  $2.22\times 10^{-4}$   &  $-$       \\
 $16\times 16\times 16 $   &   $1.15\times 10^{-4}$  &  $0.95$  & &  $1.15\times 10^{-4}$   &  $0.95$  \\
 $24 \times 24\times 24$  &  $6.54\times 10^{-5}$ &  $1.4$    & &  $6.54\times 10^{-5}$   &  $1.4$    \\
 $32\times 32 \times 32$  &  $4.01\times 10^{-6}$  &  $1.7$    & & $4.01\times 10^{-5}$    &   $1.7$   \\
\hline
\end{tabular}
\end{center}
\caption{\label{grconv_heat} Error convergence of the ParaDIn and sequential BDF1 schemes for the 2-D nonlinear heat equation.
}
\end{table}
\begin{figure}[!h] 
   \begin{center}
                  \includegraphics[width=0.75\linewidth]{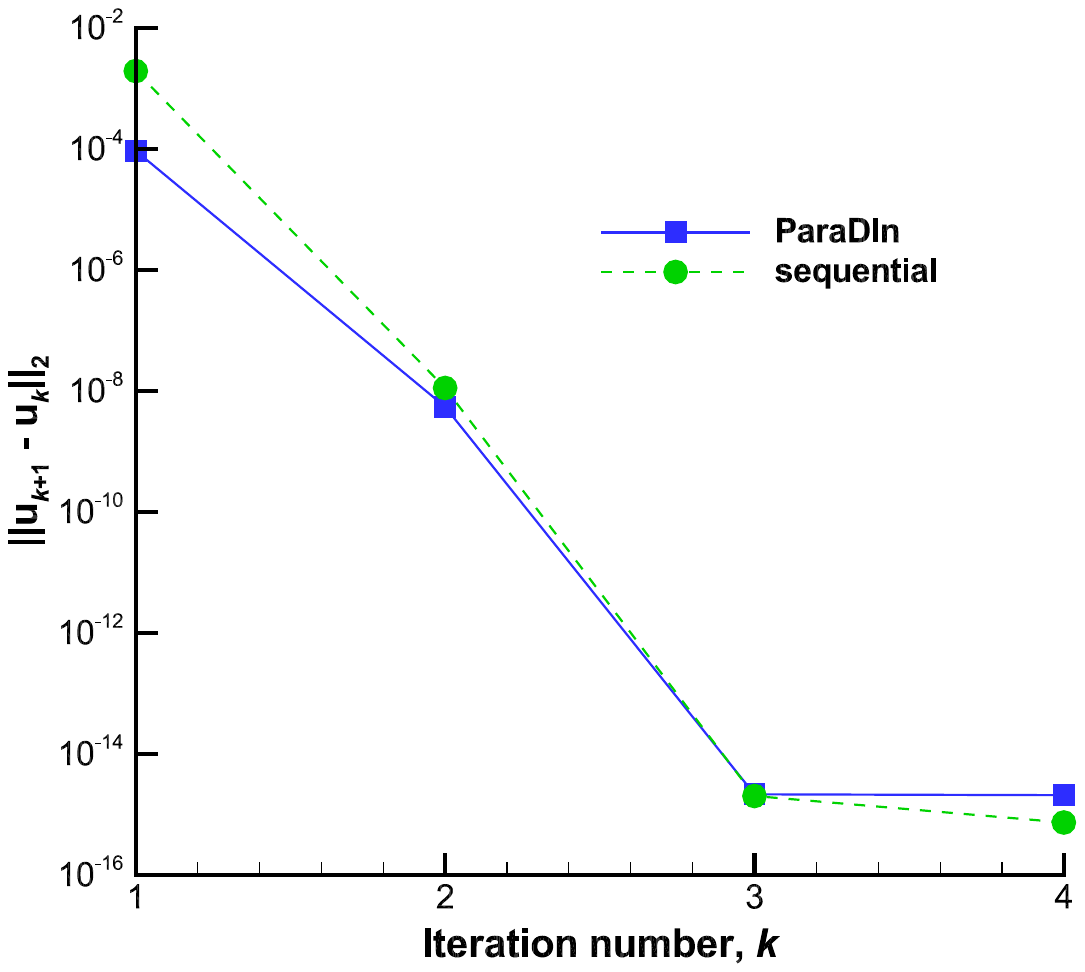} 
  \end{center}
	 \caption{Convergence histories of Newton iterations obtained with the parallel-in-time and sequential BDF1 schemes on the grid with $N_t=8$ and $N_x=N_y=64$ for the 2-D nonlinear heat equation.}  
\label{convhist_heat}
\end{figure}

\subsection{2-D nonlinear heat equation}
\label{nheat}
The first test problem is a nonlinear heat equation (\ref{eq:CLE}) with  $f(u) = g(u)= 0,   \forall (x,y)\in[0.1,1.1]\times[0.1,1.1],  \ t\in[0,1]$ The viscosity coefficient in Eq. (1) is given by $\mu = \mu_0 u^2$, where $\mu_0$ is set equal to $10^{-6}$. This nonlinear heat equation has the following exact solution:
$$
u_{\rm ex}(x,t) = \sqrt{\sqrt{\frac \alpha{\mu_0}}(x + y) + \alpha t  +1},
$$
where $\alpha$ is a positive constant that is set to be 1 for all test cases considered in this section.  The above exact solution is used to define the initial and boundary conditions given by Eq. (\ref{eq:BC}).  An initial guess for the ParaDIn method is the solution of this parabolic equation obtained by the corresponding sequential scheme on a grid coarsened by a factor of $c_f = 4$ both in each spatial direction and time, which is then  interpolated to the original grid.
\begin{table}[!h]
\begin{center}
\begin{tabular}{cccccc}
\hline
Number     &  $N_t$  & Serial           &   Parallel         &  Speedup  &  Parallel     \\
of cores     &               & runtime (s)  &  runtime  (s)   &                   &  efficiency  \\
\hline
  4   &  4  &   1112 .8 &   286.3     &   3.9    &  97\% \\
  8   &  8  &   1489.8 &   194.0      &  7.7     &  96\% \\
 16  &  16 &  2971.9   &   199.2     &   14.9   &  93\% \\
 24 &  24 &  4466.4  &   204.6     &  21.8   &  91\% \\
 32 &  32 &  5969.1  &   209.4     &  28.5    &  89\% \\
%
\hline
\end{tabular}
\end{center}
\caption{\label{speedup_heat} The speedup and parallel efficiency of the ParaDIn-BDF1 scheme for the 2-D nonlinear heat equation on grids with  $N_t = 4, 8, 16, 24, 32$,    and $N_x = N_y = 64$.
}
\end{table}

First,  we evaluate the accuracy and convergence properties of the proposed ParaDIn method.  
The $L_{\infty}$ discretization error and its convergence rate computed over the entire space-time domain  with the ParaDIn and sequential schemes on a sequence of globally refined grids are presented in Table \ref{grconv_heat}. Note that for this test problem, the discretization error is dominated by its spatial component, so that the BDF1 scheme demonstrates the convergence rate higher than one.
Convergence histories of the Newton iterations of the ParaDIn and sequential BDF1 schemes obtained on the same grid with $N_t=8$ and $N_x=N_y=64$ are compared in Fig.  \ref{convhist_heat}.  For both the parallel and sequential BDF1 schemes, the $L_2$ residual norm at each Newton iteration is  computed over all time levels as follows:
$$
\| {\bf u}_{k+1} - {\bf u}_k \|_2 =\sqrt{\frac1{N_t N_s}\sum\limits_{n=1}^{N_t}\sum\limits_{i=1}^{N_x}\sum\limits_{j=1}^{N_y}\left[(u^n_{ij})_{k+1} - (u^n_{ij})_{k}\right]^2},
$$
where $k$ is the Newton iteration index.   As evident from this comparison, both methods converge quadratically except for the last 4-th iteration which is affected by the roundoff error.  

To evaluate the parallel performance of the ParaDIn method described in Section \ref{PITBDF1},  we evaluate the weak scaling behavior of the ParaDIn-BDF1 scheme by using one time step per core and increasing the number of cores together with the number of time steps, while the spatial grid remains unchanged.  As a result, the computational cost of one time step stays nearly the same for each core, while the communication time grows, as the number of time steps increases. 
Table \ref{speedup_heat} shows the speedup and parallel efficiency of the ParaDIn-BDF1 scheme for the 2-D nonlinear heat equation on a sequence of space-time grids with $N_t = 8, 16, 24, 32$ and $N_x = N_y = 64$.  For all grids presented in Table \ref{speedup_heat},  Newton iterations of the ParaDIn method  converge quadratically and only 2-3 iterations are needed to drive the $L_2$ norm of the residual below the tolerance that is set to be $1/10$ of the true discretization error on the corresponding grid.  As one can see from this comparison, the proposed ParaDIn-BDF1 scheme provides nearly ideal speedups for all numbers of cores used.  Furthermore, the parallel efficiency  is very close to its optimal value on coarse temporal grids and gradually decreases to $88\%$ as the number of cores increases to 32, which is due to a communication overhead between the computing cores.  
\begin{table}[!h]
\begin{center}
\begin{tabular}{ccccccc}
\hline
Number &  $N_t$  & Sequential  &Number of   &  ParaDIn         & Number of  &  Speedup     \\
of cores &               & runtime (s) &iterations    &  runtime (s)     & iterations   &                      \\
\hline
   4  &  4  &  1103.7    &   3.0  & 286.3   & 3  & 3.9  \\
  8   &  8  &  1475.2   &   2.0  &  194.0   & 2   & 7.6  \\
 16  &  16 &  2936.5  &   2.0  &  199.2  & 2   & 14.7 \\
 24 &  24 &  4414.5  &  2.0   &  204.6  & 2   & 21.6 \\
 32 &  32 &  5900.7  &  2.0  &  209.4  & 2   & 28.2\\
%
\hline
\end{tabular}
\end{center}
\caption{\label{serial_heat} Runtimes of the parallel and sequential BDF1 schemes for the 2-D nonlinear heat equation on grids with $N_t = 8, 16, 24, 32$,  and $N_x = N_y = 64$.
}
\end{table}

We also compare the computational times taken to solve this 2-D nonlinear heat equation with the sequential and parallel-in-time BDF1 schemes on the same sequence of uniform grids in Table~\ref{serial_heat} .  
As follows from this comparison,  both the ParaDIn method and the corresponding sequential scheme converge quadratically in two Newton iterations except for $N_t=4$ that required three iterations. Furthermore, the new parallel-in-time BDF1 scheme  demonstrates nearly optimal speedup as compared with the sequential counterpart.  It should also be noted that the ParaDIn and sequential BDF1 solutions are nearly identical to each other on all grids considered.
\begin{table}[!h]
\begin{center}
\begin{tabular}{cccccc}
\hline
                                                &  Sequential BDF1 &   & \vline &  ParaDIn BDF1 &  \\
\hline
 $N_t\times N_x \times N_y$ & $ L_1$ error & $L_1$ rate  &\vline& $ L_1$ error & $L_1$ rate \\
\hline
   $6\times 6\times 6     $   &  $2.87\times 10^{-2}$  & $ - $      & &  $2.87\times 10^{-2}$   &  $-$       \\
 $12\times 12\times 12 $   &   $1.56\times 10^{-2}$  &  $0.88$  & &  $1.56\times 10^{-2}$  &  $0.88$  \\
 $18\times 18\times 18 $   &   $1.05\times 10^{-2}$  &  $0.97$  & &  $1.05\times 10^{-2}$   &  $0.97$  \\
 $24 \times 24\times 24$  &  $7.91\times 10^{-3}$   &  $1.00$  & &  $7.91\times 10^{-3}$   &  $1.00$    \\
 $30\times 30 \times 30$  &  $6.26\times 10^{-3}$  &  $1.05$  & & $6.26\times 10^{-3}$   &   $1.05$   \\
\hline
\end{tabular}
\end{center}
\caption{\label{grconv_Burgers} Error convergence of the ParaDIn and sequential BDF1 schemes for the 2-D Burgers equation.
}
\end{table}
\begin{figure}[!h] 
   \begin{center}
                  \includegraphics[width=0.75\linewidth]{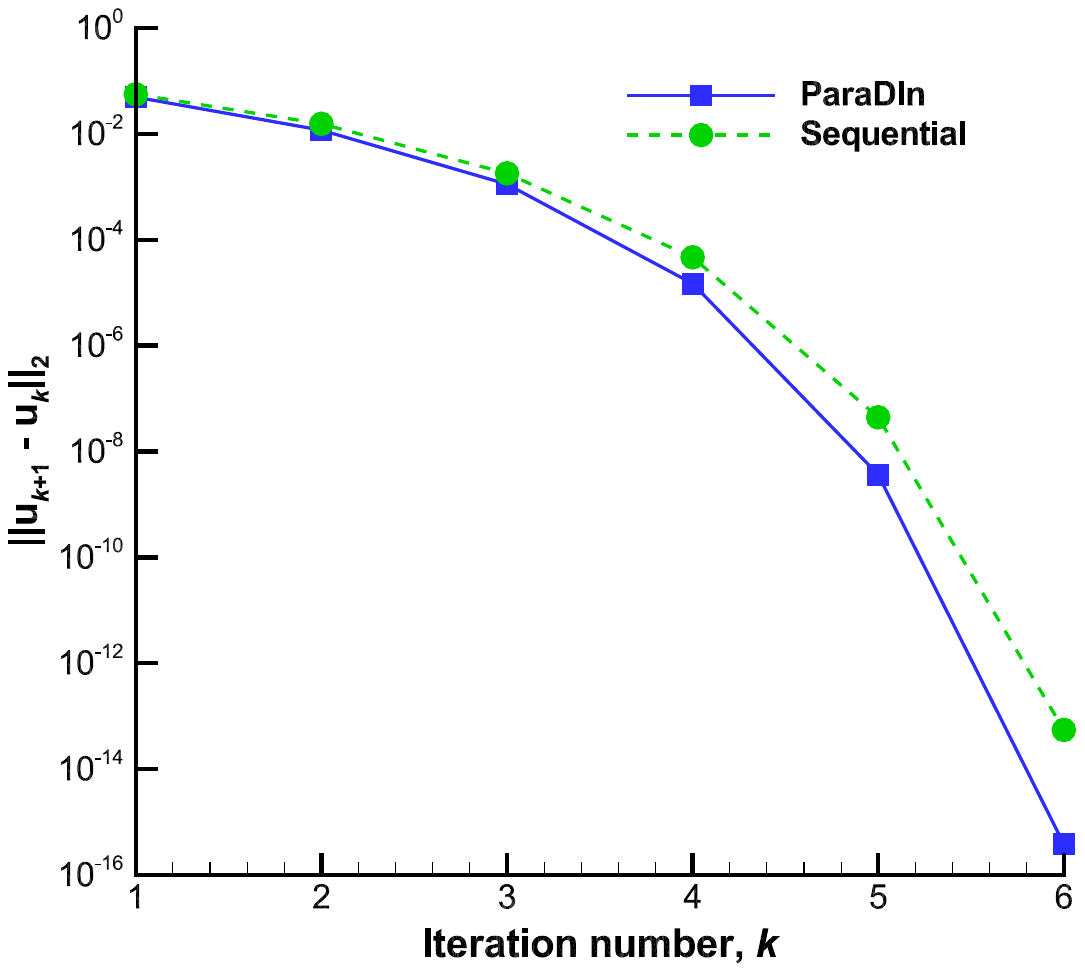} 
  \end{center}
	 \caption{Convergence histories of Newton iterations obtained with the parallel-in-time and sequential BDF1 schemes on the grid with $N_t=12$ and $N_x=N_y=63$ for the 2-D Burgers equation.}  
\label{convhist_Burgers}
\end{figure}

\subsection{2-D viscous Burgers equation}
The second test problem is  the 2-D viscous Burgers equation given by Eq. ~(\ref{eq:CLE}) with  $f(u) = g(u)=u^2/2.$ For this problem, we use the following exact solution:
\begin{equation}
\label{eq:Burgers}
u_{\rm ex} = \frac v2 \left[ 1 - \tanh\left( \frac{v(x+y-vt)}{4\mu}\right) \right],
\end{equation}
where $v$ is a shock speed that is set to be $0.5$. The viscosity coefficient $\mu$ is constant and set equal to $10^{-3}$, so that the viscous shock wave given by Eq. (\ref{eq:Burgers}) is not fully resolved on any grid considered and behaves as a strong discontinuity.  For the Burgers equation,  an initial guess for the ParaDIn method is constructed the same way as in the previous test problem with the coarsening factor of $c_f=3$ in each spatial direction and time.

The $L_1$ discretization error and its convergence rate  obtained with the ParaDIn and corresponding sequential methods are presented in Table \ref{grconv_Burgers}.  As expected, the parallel-in-time and sequential BDF1 schemes show an identical error convergence whose rate approaches to one as the grid is refined.
We compare convergence histories of Newton iterations obtained with the ParaDIn and sequential BDF1 schemes on the same grid with $N_t=12$ and $N_x=N_y=63$ in Fig. ~\ref{convhist_Burgers}.  Similar to the previous test problem,  the sequential and parallel-in-time BDF1 schemes  demonstrate a very similar convergence behavior for the Burgers equation.  Note,  however,  that both methods do not converge quadratically except for the last two iterations, as one can see in Fig. ~\ref{convhist_Burgers}.
\begin{table}[!h]
\begin{center}
\begin{tabular}{cccccc}
\hline
Number  &  $N_t$  & Serial ParaDIn  &   ParaDIn        &  Speedup  &  Parallel     \\
of cores  &               & runtime (s)  &  runtime  (s)   &                   &  efficiency  \\
\hline
  6  &   6  &  1385.2  &   242.2     &    5.7    &  95\% \\
  12 &  12  &  2846.7 &   253.9     &    11.2   &  93\% \\
 18  &  18 &  4239.4  &   259.7     &   16.3   &  91\% \\
 24 &  24 &  5639.6  &   262.0     &  21.5    &  90\% \\
 30 &  30 &  9436.9  &   355.6     &  26.5    & 88\% \\
%
\hline
\end{tabular}
\end{center}
\caption{\label{speedup_Burgers} The speedup and parallel efficiency of the ParaDIn-BDF1 scheme for the 2-D Burgers equation on grids with $N_t = 6, 12,  18,24, 30$,  and $N_x = N_y = 63$.
}
\end{table}
%
%
\begin{table}[!h]
\begin{center}
\begin{tabular}{ccccccc}
\hline
Number &  $N_t$  & Sequential  &Number of   &  ParaDIn         & Number of  &  Speedup     \\
of cores &               & runtime (s) &iterations    &  runtime (s)     & iterations   &                      \\
\hline
  6   &  6  &  1828.8   &   4.0  & 242.2  & 3   & 7.6   \\
 12  &  12 &  2891.9   &   3.08&  253.9 & 3   & 11.4 \\
 18  &  18 &  4178.5   &   3.0  &  259.7  & 3   & 16.1 \\
 24 &  24 &  5574.8   &  3.0   &  262.0 & 3   & 21.3 \\
 30 &  30 &  6993.3  &  3.0  &  355.6  & 4   & 19.7 \\
\hline
\end{tabular}
\end{center}
\caption{\label{seq_Burgers} Runtimes of the parallel-in-time and sequential BDF1 schemes for the 2-D Burgers equation on grids with $N_t = 6, 12, 18, 24, 30$,  and $N_x = N_y = 63$.
}
\end{table}

We evaluate the parallel performance of the ParaDIn method for unsteady problems with strong discontinuities by solving  the 2-D Burgers equation on a sequence of uniform rectangular grids with $N_t = 6, 12, 18,  24, 30$ and $N_x = N_y = 63$.    The speedup and parallel efficiency obtained with the ParaDIn-BDF1 scheme are presented in Table~\ref{speedup_Burgers}.  Similar to the previous test problem,  the new time-parallel BDF1 scheme demonstrates speedups which are close to their maximum possible values with a slight decrease in the parallel efficiency as the temporal grid is refined.  
Computational times required to solve the 2-D Burgers equation until $T_f=1$ by the sequential and parallel-in-time BDF1 schemes are presented in Table \ref{seq_Burgers}.  Since the viscous shock wave at $\mu=10^{-3}$ is not resolved on any of these grids, the local- and global-in-time Newton iterations associated with the sequential and parallel-in-time BDF1 schemes, respectively,  do not converge quadratically.  Despite this suboptimal convergence, both methods reduce the $L_2$ norm of the residual below a tolerance (which is set to be $1/10$ of the true discretization error) in about 3 Newton iterations.  On the coarsest grid with $N_t=6$,  the Newton method of  the sequential algorithm in average converges by one iteration slower than that of the ParaDIn method.  This slower convergence of the sequential algorithm is due to the fact that its initial guess, which is the solution at previous time step,  is less accurate  than the initial guess of the ParaDIn method, which is a coarse-grid solution obtained on the space-time grid with coarsening factor of $c_f = 3$.  Note that this convergence behavior reverses on the finest grid ($N_t=30$), because the time step size becomes sufficiently small,  thus providing a more accurate initial guess for the Newton method of the sequential algorithm at each time step.  
As expected,  the numerical solution obtained with the parallel-in-time BDF1 scheme converges to that of its sequential counterpart for all grids considered, because the decoupled system of equations (\ref{eq:Dec}) of the ParaDIn method is equivalent to the original fully coupled equations (\ref{eq:LocalNewton}-\ref{eq:GlobalNewton}) of the sequential algorithm.

\section{Conclusions}
In this paper, we have introduced and developed a new strategy for parallel-in-time integration of unsteady nonlinear differential equations  discretized by the implicit first-order backward difference (BDF1) scheme in time.
The global all-at-once system of nonlinear discrete equations is solved by the Newton method for all time levels simultaneously.  
The key idea of the proposed methodology, which is herein referred to as ParaDIn, is to recast the block bidiagonal system of linear equations at each Newton iteration in a form that is equivalent to its inverse.   As a result,   this system of equations is decoupled in time into independent blocks that can be solved in parallel.  We have proven that the decoupled and original all-at-once systems are equivalent to each other. Therefore, the ParaDIn method preserves the quadratic convergence rate of Newton iterations and provides a nearly ideal speedup for nonlinear problems with both smooth and discontinuous solutions.  For the proposed parallel-in-time BDF1 scheme,  each time level  is computed on its own computing core regardless of the number of degrees of freedom used for discretization of the spatial derivatives.  As has been shown in section 7,  some upper bound should be imposed on the total number of time steps that can be integrated in parallel by using the ParaDIn method.  This constraint is due to the fact that the condition number increases together with the number of time steps.
A preconditioning strategy  designed for circumventing this constraint will be presented in our future paper.  The numerical results show that the proposed ParaDIn-BDF1 scheme demonstrates a nearly ideal speedup and a parallel efficiency on up to 32 computing cores for the 2-D nonlinear heat and Burgers equations with smooth and discontinuous solutions.

\bigskip  
\noindent{\large\bf Acknowledgments}


The first author acknowledges the support from Army Research Office and Department of Defense through grants
W911NF-17-0443 and W911NF-23-10183.






\end{document}